\documentclass[11pt]{amsart}
\theoremstyle{plain}
\newtheorem{Thm}{Theorem}

\newtheorem{Coro}[Thm]{Corollary}
\newtheorem{Prop}[Thm]{Proposition}
\newtheorem{Lem}[Thm]{Lemma}

\begin{document}
\large
\title[Properties of positive solutions]
{Properties of positive solutions of an Elliptic Equation with
negative exponents}

\author{Li Ma, Juncheng Wei}

\address{Li Ma, Department of mathematical sciences \\
Tsinghua university \\
Beijing 100084 \\
China}

\email{lma@math.tsinghua.edu.cn}

\address{J.C.Wei,Department of mathematics \\
The Chinese university of Hong Kong \\
Shatin, Hong Kong} \email{wei@math.cuhk.edu.hk}

\dedicatory{}
\date{Oct 8th, 2006}

\thanks{The research is partially supported by the National Natural Science
Foundation of China 10631020, by the grant KZ200710025012 of
Beijing Education committee, and by the grant SRFDP 20060003002 of
Ministry of national Education of China. The research of the
second named author is partially supported by an Earmarked Grant
from RGC of Hongkong.}

\begin{abstract}
In this paper, we study the existence and non-existence result of
 positive solutions to a singular elliptic equation with negative power
 on the bounded smooth domain or in the whole
 Euclidean space. Our model arises in the study of the steady states of thin films
 and other applied physics.
 We can get some useful local gradient estimate and L1 lower bound for positive
 solutions of the elliptic equation. A uniform positive lower bound for
 convex positive solutions is also obtained. We show that in lower
 dimensions, there is no stable positive solutions in the whole space.
 In the whole space of dimension two, we can show that there is no positive smooth solution
 with finite Morse index. Symmetry properties of related integral
 equations are also given.
\end{abstract}

 \maketitle

\begin{section}{Introduction}
  In this paper, we mainly study some property
  of positive solutions of the
  elliptic equation in the domain $\Omega\subset\mathbf{R}^{n}$
     \begin{equation}\label{eqn1}
  \Delta u=u^{\tau},\ \ in \ \ {\Omega},
  \end{equation}
 where $\tau<0$. The elliptic equation comes
 from applied physics, mathematical biology, and the steady states of thin film model
  of viscous fluids:
 \begin{equation}\label{film}
u_t=-div(f(u)\nabla \Delta u)-div(g(u)\nabla u).
  \end{equation}
One may see \cite{BP98}, \cite{LP00},\cite{JN}, \cite{Mea}, and
\cite{DP01} for more background and interesting results about
(\ref{film}).

 Using the maximum principle (see \cite{E94} and \cite{Ma06}),
 we can get a useful gradient
estimate for positive solutions of the elliptic equation
(\ref{eqn1}).

Namely, we have

\begin{Thm}\label{thm3} Assume $\tau\leq 0$,
and $\Omega\subset \mathbf{R}^n$.
Let $u\in C^2(\Omega)$ be a positive solution to the equation
(\ref{eqn1}) in $\Omega$.
  Then for any $R>0$, $x_0\in\Omega$, and $x\in B_R(x_0)\subset\Omega$
  we have absolute constant $C=C(R)$ such that
   \begin{equation}\label{gradient}
|\nabla u(x)|^2\leq Cu(x)^2+u(x)^{1+\tau}.
   \end{equation}
  In particular, when $\tau=-1$, we have
  $$
|\nabla u(x)|^2\leq Cu(x)^2+1.
  $$
\end{Thm}
As a consequence of the local gradient estimate in Theorem
\ref{thm3}, we have the following improvement of Theorem 7.1 in
\cite{Mea}.
\begin{Coro} For any sequence $(u_j)$ of positive solutions to the equation
(\ref{eqn1}) with $\tau=-1$ and with boundary data $\phi_j\leq M$,
and for every compact sub-domain $K$ of $\Omega$, there is a
constant $C=C(n,K,\Omega)$ such that
 \begin{equation}\label{grad2}
|\nabla u_j|\leq (M+1)C, \; on \; K.
 \end{equation}
Hence, the limit $u$ of any convergent subsequence of $(u_j)$ is a
Lipschitz continuous weak solution to a free boundary problem of
the equation
\begin{equation}\label{free2}
\Delta
u=u^{\tau}\chi_{\{u>0\}}.
\end{equation}
\end{Coro}
For the proof of this corollary, we just note that the solution
$u$ is a subharmonic function and it attains its maximum only on
the boundary $\partial\Omega$. Then we use the gradient bound to
get the conclusion. We remark that the result is optimal in the
sense that $u$ is not differentiable at its zero point as noticed
in \cite{Mea}. It is unclear how large the zero level set
$\mathbf{S}(u)=\{u=0\}$ is. However, a general study was made in
the paper of Jiang and F.H.Lin \cite{JL}.

Note that the gradient estimate implies the Harnack inequality for
positive solutions with $u(x)\geq 1$ for all $x\in B_R(x_0)$. One
can use this fact in deriving compactness result.

We shall discuss the existence theory of positive solutions to
(\ref{eqn1}) in section \ref{EXIST}.

 Using the testing function
method, we have the following very useful L1 lower bound result.
\begin{Thm}\label{thm4} Assume $\tau\leq0$ and assume that $\Omega\subset \mathbf{R}^n$
is an open subset in $\mathbf{R}^n$. Let
$f:\mathbf{R}\to\mathbf{R}_+$ be a positive function such that
$$
s^{\frac{\tau}{\tau-1}}f(s)^{\frac{1}{1-\tau}}\geq C_0, \; for \;
s>0
$$
for some constant $C_0$. Let $u\in C^0(\Omega)$ be a positive weak
solution to the equation  \begin{equation}
\label{eqn10}
\Delta
u=f(u)
\end{equation}
in $\Omega$.
  Then for any $R>0$ and $x_0\in \Omega$ ( with $B_R(x_0)\subset \Omega$),
  we have absolute constant $C(n,\tau)$ such that
 \begin{equation}\label{L1}
\int_{B_R(x_0)}u\geq C(n,\alpha)R^{n+\frac{2}{1-\tau}}.
\end{equation}
\end{Thm}
We remark that we shall use Theorem \ref{thm4} for $f(u)=u^{\tau}$
and in this case, we can use the convexity of $f(u)$, the
spherical average method, and an ODE comparison lemma to get the
$L^1$ lower bound. However, we shall give a proof which can be
used on manifolds.

There are many consequences of Theorem \ref{thm4}, and they will
be discussed in section \ref{CONS}.

We now try to set up a global upper bound for positive solutions
to (\ref{eqn1}) in $\mathbf{R}^n$.

\begin{Thm}\label{thm11} Assume $\tau\leq0$.
Let $1\leq u\in C^2(\mathbf{R}^n)$ be a positive solution to the
equation (\ref{eqn1}) in $\mathbf{R}^n$.
  Then
  we have absolute constant $C(n)$ such that
 \begin{equation}\label{bound2}
u(x)\leq C(n)(|x|^2+1),
\end{equation}
and
\begin{equation}\label{bound3}
|\nabla u(x)|\leq C(n)(|x|+1),
\end{equation}
for all $x\in \mathbf{R}^n$.
\end{Thm}

The important consequence of Theorem \ref{thm11} is the following
possibly well-known result.
\begin{Prop} Assume that $u>0$ is a positive solution to the
equation
$$
\Delta u=1, \; \; in \; \; \mathbf{R}^n.
$$
Then $u$ is a polynomial of the form
$$
a_0+\sum_{j=1}^n a_jx_j^2,
$$
where $a_0>0$,$a_j\geq 0$ for $j=1,...n$, and $2\sum_ja_j=1$.
\end{Prop}
\begin{proof} We may assume $u(x)\geq 1$ by considering
 $u(x)+1$ if necessary. Using our upper bound estimate in Theorem
\ref{thm11}, we know u has at most quadratic growth. Consider
$w(x)=u(x)-\frac{|x|^2}{2n}$. Then $w$ is a harmonic function with
quadratic growth. Hence, $w$ and $u$ is a quadratic polynomial,
which gives the conclusion.
\end{proof}

 We shall try to study
a Liouville property of positive solutions to the equation
(\ref{eqn1}) in the whole space or in the half space.
\begin{Thm}\label{Prop2} Assume that $n>2$ and $\tau\leq -1-\frac{2n}{n-2}$.
We have the\emph{ Liouville property} that there is no positive
convex solution $u(x)$ to (\ref{eqn1}) both on the whole space
$\mathbf{R}^n$ and on the half space $\mathbf{R}^n_+$ with
$u(x)\geq 1$ everywhere.
\end{Thm}

 Using this Liouville property, we can prove the following compactness result,
 which can be considered as a generalization
 of Proposition \ref{thm5}.
\begin{Thm}\label{thm6}
Assume that $n> 2$ and $\tau\leq-1-\frac{2n}{n-2}$. Let $\Omega$
be a bounded or unbounded smooth convex domain. Let $u$ be a
positive convex solution to (\ref{eqn1}) on $\Omega$. Then, we
have a uniform constant $C=C(\Omega)$ such that
$$
u(x)\geq C, \; \; x\in\Omega.
$$
\end{Thm}
Note that the convexity property in Theorem \ref{thm6} can not be
removed since $u(x)=A|x|^{\frac{2}{1-\tau}}$ is a non-negative
solution to (\ref{eqn1}) with
$$
A=[\frac{(1-\tau)^2}{2(n+1)-2(n-1)\tau}]^{\frac{1}{1-\tau}}.
$$

From the variational point of view, it is very interesting to
discuss positive solutions with finite Morse index to the
following equation
\begin{equation}\label{Morse}
 \Delta u=u^{\tau},\ \ in \ \ {\mathbf{R}^n},
\end{equation}
where $\tau<0$, with finite Morse index. Assume that $u\in C^2$ is
a positive solution to (\ref{Morse}). Define
$$ E(\phi)=\int_{\mathbf{R}^n}(|\nabla \phi|^2+\tau
u^{\tau-1}\phi^2),
$$
where $\phi\in C^2_0(\mathbf{R}^n)$. By definition, we say the
positive solution $u$ to (\ref{Morse}) with \emph{finite Morse
index} $k$ if there exist $L^2$ orthogonal nontrivial functions
$\{\phi_j\}_{j=1}^k\subset C^2_0(\mathbf{R}^n)$ such that we have
$E(\phi)<0$ for $\phi\in \mathbf{W}:=span\{\phi_j\}-\{0\}$, and
$E(\phi)\geq 0$ for $\phi\perp\mathbf{W}$. If $k=0$, we say that
the solution $u$ is \emph{stable}. When $n=2$, we can prove there
are no finite Morse index solutions.

\begin{Thm}
\label{morse} Assume $\tau<0$. There is no finite Morse index
positive solutions to (\ref{Morse}) for $n=2$.
\end{Thm}
This result should have application to bifurcation theory to the
equation (\ref{eqn1}).

Here is the plan of our paper.
 In last two sections, we
discuss symmetry properties of related integral equations. We
prove Theorem \ref{thm3} in section \ref{GRAD} and prove Theorem
\ref{thm4} in section \ref{LOW}. In section \ref{UPP}, we prove
Theorem \ref{thm11}. Theorem \ref{thm6} is proved in section
\ref{COM}. The Liouville property is proved in section \ref{LIOU}.
The \ref{DISC}-th section is about some discussions to the
existence problem and upper bound for positive solutions to
(\ref{eqn1}). The proof of Theorem \ref{morse} and some bound
about positive solution with finite Morse index are given in
section \ref{FINI}. As we mentioned before, we shall discuss the
existence theory of positive solutions to (\ref{eqn1}) in section
\ref{EXIST}. Many consequences of Theorem \ref{thm4} will be
discussed in section \ref{CONS}.

  In the following, we shall use $C$ to denote different constants which depend
  only on $n$, $\tau$, $\mu$, and the solution $u$ in
  varying places.
\end{section}

\begin{section}{proof of Theorem \ref{thm3}}\label{GRAD}

In this section, we consider the gradient estimate of positive
solutions to the elliptic equation
\begin{equation}\label{mu2}
\Delta u=u^{\tau}, \; \; in \; \; \Omega\subset\mathbf{R}^n,
\end{equation}
where $\tau<0$. We shall use the maximum principle method
\cite{E94} (see also \cite{Ma06} and \cite{Ne}).

Recall the following \emph{basic formula}.
 Let $v$ be any smooth
function in $\mathbb{R}^n$. Then, we have
$$
\frac{1}{2}\Delta |\nabla v|^2=(\nabla \Delta v, \nabla v)+|D^2v|^2,
$$
which can be proved by an elementary calculation.

Let us begin with the gradient estimate for positive solution $u$ to
the following general elliptic equation:
$$
\Delta u=f(u), \; \; in\; \; \mathbf{R}^n.
$$
In our case we shall set $f(u)=u^{\tau}$. Set
$$
w=\log u.
$$
Then we have $$ \nabla w=u^{-1}\nabla u,
$$
and we can get
$$
\Delta w=-|\nabla w|^2+F(w),
$$
where
$$
F(w)=u^{-1}f(u)=e^{-w}f(e^w).
$$
In particular, $F(w)=e^{(-1+\tau)w}$ and $F'(w)=(-1+\tau)F(w)$ for
our case.

Assume $R_2>R>0$. Let $\phi$ be a cut-off function in $B_0(R_2)$
with $\phi=1$ on $B_0(R)$. Define
$$
P=\phi |\nabla w|^2,
$$
which is usually called the Harnack quantity for the solution $u$.

At the maximum point of $P$, we have the first order condition
$$ \nabla P=0,$$ which implies that
$$
\nabla |\nabla w|^2=-\phi^{-2}\nabla\phi P.
$$
and the second order condition $(\star)$:
$$
0\geq\Delta P=P_0(\phi)P+\phi \Delta |\nabla w|^2,
$$
where
$$
P_0(\phi)=\Delta\phi-2|\nabla\phi|^2\phi^{-2}.
$$

 Using the \emph{basic formula}, we have
$$
\phi \Delta |\nabla w|^2=2\phi |D^2w|^2+2\phi (\nabla\Delta w,\nabla
w).
$$
Note that
$$
\phi |D^2 w|^2\geq \frac{2\phi}{n}|\Delta
w|^2=\frac{2}{n\phi}(-P+\phi F(w))^2,
$$
and
$$
2\phi (\nabla\Delta w,\nabla w)\geq 2F'P-2\phi (\nabla |\nabla
w|^2,\nabla w)=2F'P-2\phi ^{-1}(\nabla\phi,\nabla w)P,
$$
and then, for any $\mu>0$,
$$
2\phi (\nabla\Delta w,\nabla w)\geq 2F'P-2\mu^{-1}\phi
^{-2}|\nabla\phi|^2P-\mu\phi^{-1}P^2.
$$

 Choose $\mu=\frac{1}{4n}$. Then
$$
2\phi (\nabla\Delta w,\nabla w)\geq 2F'P-4n\phi
^{-2}|\nabla\phi|^2P-\frac{1}{4n}\phi^{-1}P^2.
$$
Hence,
$$
\phi \Delta |\nabla w|^2\geq \frac{2}{n\phi}(-P+\phi F(w))^2+
2F'P-4n\phi ^{-2}|\nabla\phi|^2P-\frac{1}{4n}\phi^{-1}P^2.
$$
Then from $(\star)$ we have
$$
A(\phi, F')P\geq \frac{2}{n}(-P+\phi F)^2-\frac{1}{4n}P^2.
$$
Here $$ A(\phi, F')=4n\phi ^{-1}|\nabla\phi|^2-2\phi F'-\phi
P_0(\phi).
$$

If $P\leq 2\phi F$, then we have
$$
|\nabla w|^2\leq 2F.
$$
We remark that in this case, we have $$ |\nabla u|^2\leq
2u^2F=2uf(u).
$$

Otherwise, we have
$$
-P+\phi F\leq -P/2\leq 0
$$
and
$$
\frac{2}{n}(-P+\phi F)^2-\frac{1}{4n}P^2\geq \frac{1}{4n}P^2.
$$
Hence, we have
$$
P\leq 4nA(\phi, F').
$$
In conclusion, we have on $B_0(R)$,
$$
|\nabla w|^2\leq max(4nA(\phi, F'),2F),
$$
which implies the conclusion of Theorem 3.

We remark that our gradient estimate can be extended to other
elliptic equation like
$$
-\Delta u=u^{\tau}.
$$

\end{section}

\begin{section}{proof of Theorem {\ref{thm4}}}\label{LOW}

Our aim of this section is to set up an energy lower bound for the
positive solution $u$ to (\ref{mu2}) in every ball $B_R(x_0)$.
Without loss of generality, we take $x_0=0$. Let $R_2=2R_1>0$ and
let $\xi(|x|)$ be a cut-off function with its support in the ball
$B_{R_2}(0)$, $\xi(|x|)=1$ on $B_{R_1}(0)$, and
$$|\nabla \xi|\leq 4/R_1, \; \;  |\Delta\xi|\leq 100/R^2_1.$$

Multiplying both sides of (\ref{eqn1}) by $|x|^2\xi$ and integrating
over the ball $B_{R_2}(0)$, we then get
$$
\int u\Delta(|x|^2\xi)=\int f(u)|x|^2\xi.
$$
Note that the right side is bigger than
$$
\int_{B_{R_1}(0)} f(u)|x|^2;
$$
and the left side is less than
$$
C\int_{B_{R_2}(0)} u,
$$
where $C>0$ is an absolute constant depending only on the dimension
$n$. That implies that
$$
\int_{B_{R_1}(0)} f(u)|x|^2\leq C\int_{B_{R_2}(0)} u.
$$
Let $p>1$. Then we have
$$
(\int_{B_{R_1}(0)} f(u)|x|^2)^{1/p}(\int_{B_{R_1}(0)}
u)^{(p-1)/p}\leq C^{1/p}\int_{B_{R_2}(0)} u.
$$
Using Holder's inequality to the left side, we get
$$
\int_{B_{R_1}(0)}|x|^{2/p}f(u)^{\frac{1}{p}}u^{(p-1)/p}\leq
C^{1/p}\int_{B_{R_2}(0)} u.
$$
Choose $p=-\tau+1$. Then by our assumption on $f$, we have
$$
\int_{B_{R_1}(0)}|x|^{2/p}\leq C^{1/p}\int_{B_{R_2}(0)} u.
$$
It is elementary to compute that
$$
\int_{B_{R_1}(0)}|x|^{2/p}=\frac{p\omega_n}{2+p(n-1)}R_1^{n+\frac{2}{p}},
$$
$\omega_n$ is the volume of the unit ball $B_1(0)$. This implies the
conclusion of Theorem \ref{thm4}.

We remark that our argument above can also be used to smooth
positive solutions to the following equation
$$
 -\Delta u=u^{\tau},\ \ in \ \ {\Omega},
$$
with $\tau\leq0$.
\end{section}

\begin{section}{proof of Theorem {\ref{thm11}}}\label{UPP}

Assume that $u(x)\geq 1$ on $B_{R}(0)$ satisfies
$$
\Delta u=u^{\tau}.
$$
Then $$ 0<\Delta u\leq 1.
$$ Using the Gradient estimate in Theorem \ref{thm3}, we can easily
get that
$$
u(x)\leq u(0) e^{CR}
$$
for all $x\in B_R(0)$. Here $C$ is a uniform constant. However,
this estimate is too rough.

We now use the mean value property to do better. Fix $0<r=|x|\leq
R/4$. Since $u>0$ is subharmonic, on one hand, we have
$$
u(x)\leq \frac{1}{|B_r(x)|}\int_{B_r(x)} u.
$$
On the other hand, let
$$
w(y)=u(y)-\frac{|y|^2}{2n}.
$$
Then
$$
\Delta w=\Delta u-1\leq 0.
$$
Hence $w$ is super-harmonic, and for $R=2r$,we have
\begin{equation}\label{upper}
1=w(0)\geq \frac{1}{|B_R(0)|}\int_{B_R(0)}
u-\frac{2|x|^2}{n(n+2)}.
\end{equation}
Note that $B_r(x)\subset B_R(0)$ and
$$
\frac{1}{|B_R(0)|}\int_{B_R(0)} u\geq
\frac{|B_r(x)|}{|B_R(0)|}\frac{1}{|B_r(x)|}\int_{B_r(x)}u\geq
\frac{1}{2^n}\frac{1}{|B_r(x)|}\int_{B_r(x)}u\geq
\frac{u(x)}{2^n}.
$$
Hence, using (\ref{upper}), we have (\ref{bound2}):
$$
u(x)\leq C_n (|x|^2+1).
$$
Using the standard interpolation argument and $0<\Delta\leq 1$, we
get the gradient growth:
$$
|\nabla u(x)|\leq C_n (|x|+1).
$$
\end{section}

\begin{section}{Compactness result}\label{COM}

In this section, we study the point-wise lower bound of positive
solutions to the equation (\ref{eqn1}).

 We now prove Theorem \ref{thm6}, which is the{ \bf
Claim:} There is a positive constant $C=C(\Omega)$ such that for
every positive convex solution to (\ref{eqn1}) with $u\in
\mathbf{S}_K$, we have
$$
u(x)\geq C, \; \; for \; \; x\in \Omega.
$$
For otherwise, we have a sequence of positive convex solutions
$\{u_j\}$ and a sequence of points $\{x_j\}\subset \Omega$ such
that
$$
u_j(x_j)=min_{\Omega}u_j(x)\to 0.
$$
Choose
$$
\lambda_j=u_j(x_j)^{\frac{1-\tau}{2}}\to 0.
$$
Set
$$
v_j(x)=\lambda_j^{-\frac{2}{1-\tau}}u_j(x_j+\lambda_jx),
$$
$$\Omega_j:=\{x\in\mathbf{R}^n; x_j+\lambda_jx\in\Omega\},
$$
and $$ B_j=B_{R\lambda_j^{-1}}(0).
$$
  Then it is elementary to see that
$$
\Delta v_j=v_j^{\tau}, \; \;in\; \; \Omega_j
$$
and
$$
v_j(x)\geq v_j(0)=1.
$$
Let $$ \widehat{\Omega}=lim_j\Omega_j.$$

 Assume that $\lambda_jd(x_j,\partial \Omega)\to\infty$. Then
$\widehat{\Omega}=\mathbf{R}^n$ and
 by
our Harnack gradient estimate and the standard Lp theory, we can
extract a convergent subsequence in $C^2(B_r(0))$ for any $r>0$,
still denoted by $\{v_j\}$, with its limit $\bar{v}$ being a
positive convex function satisfying
$$
\Delta \bar{v}=\bar{v}^{\tau}, \; \; in \; \;\mathbf{R}^n,
$$
and
$$
\bar{v}(x)\geq 1=v(0).
$$

If $\lambda_jd(x_j,\partial \Omega)\leq C$ for some constant $C$,
then we have $\widehat{\Omega}=\mathbf{R}^n_+$ and we can get a
positive solution $v\in C^2(\mathbf{R}^n_+)$, i.e.,
$$
\Delta \bar{v}=\bar{v}^{\tau}, \; \; in \; \;\mathbf{R}^n_+,
$$
and
$$
\bar{v}(x)\geq 1=\bar{v}(0).
$$
However, both cases give us a contradiction by our Theorem
\ref{Prop2}. Then, we have proved Theorem \ref{thm6}.
\end{section}

\begin{section}{Liouville property}\label{LIOU}

We prove Theorem \ref{Prop2} in this section. Choose any positive
number $k>1$ and let $\Omega=\{x\in\mathbf{R}^n; u(x)\leq k\}$. By
our assumption, $\Omega$ is a bounded convex domain. Recall the
Pohozaev identity formally. Let $g(u)=-u^{\tau}$ and let
$$G(u)=\frac{1}{1+\tau}[k^{1+\tau}-u^{1+\tau}].$$
Multiplying by $x\nabla u$ to the equation
$$
-\Delta u=g(u),
$$
we have $$ 0=div(\nabla u(x\nabla u)-x\frac{|\nabla
u|^2}{2}+xG(u))+\frac{n-2}{2}|\nabla u|^2-nG(u).
$$
 Note that $1+\tau<0$. By
integrating the equation above over $\Omega$, we get
$$
\int_{\Omega}\frac{n-2}{2}|\nabla
u|^2+\frac{n}{1+\tau}[u^{1+\tau}-k^{1+\tau}]
+\frac{1}{2}\int_{\partial\Omega}\partial_{\nu}u(x\cdot\nabla
u)=0.
$$
Note that by multiplying by $u$ to the equation, we have
$$
\int_{\Omega}|\nabla u|^2-\int_{\partial\Omega}
u\partial_{\nu}u=\int_{\Omega}u^{1+\tau}.
$$
Hence, using $\partial_{\nu}u>0$ on the boundary $\partial\Omega$,
we have
$$
[\frac{n-2}{2}+\frac{n}{1+\tau}]\int_{\Omega}u^{1+\tau}<\frac{nk^{1+\tau}|\Omega|}{1+\tau}<0.
$$
By this we get a contradiction, and then Theorem \ref{Prop2} is
true. So we are done.

\end{section}

\begin{section}{Discussions}\label{DISC}

  Let $-1<\tau<0$. Given a positive data $\phi$ on the bounded smooth domain
  $\Omega$.
  Consider the boundary problem of positive solutions to (\ref{eqn1})
  on $\Omega$ with the boundary condition $u=\phi$ on $\partial\Omega$.
Let
$$
J(u)=\frac{1}{2}\int_{\Omega}|\nabla
u|^2+\frac{1}{1+\tau}\int_{\Omega}u^{1+\tau}
$$
on the space $\mathbf{A}=\{u\in H^1(\Omega); u=\phi \; on \;
\partial\Omega\}$. Since
$$
\int_{\Omega}u^{1+\tau}\leq
|\Omega|^{-\tau}(\int_{\Omega}u)^{1+\tau},
$$
we can get a non-negative minimizer of $J(\cdot)$ on $\mathbf{A}$.
For such a minimizer, one need to handle with how large for its
zero set. One may see \cite{Ph} and \cite{GW} for more.

We now discuss some regularity result for weak solution to
(\ref{eqn1}). We only need to get upper bound for positive weak
solutions to (\ref{eqn1}) for any $\tau<0$ by assuming a positive
lower bound.
 As in the proof of Theorem \ref{thm4}, we take $R>\rho>0$
and a cut-off function $\xi=\xi(|x|)$ such that $|\nabla\xi|\leq
\frac{4}{R-\rho}$, and $\xi=1$ on $B_{\rho}$. Then using $u(x)\geq
1$, we have as before that
$$\int u\Delta(|x|^2\xi)=\int u^{\tau}|x|^2\xi\leq \int |x|^2\simeq R^{n+2}.$$
Using $\Delta|x|^2=2n$ we have
$$
\int_{B_{\rho}} u\leq
\frac{AR^2}{(R-\rho)^2}\int_{T_{R,\rho}}u+BR^{n+2},
$$
where $A,B$ are uniform constants and $$ T_{R,\rho}=B_R-B_{\rho}.
$$

 We can also derive some other interesting bound without the
point-wise lower bound.

 Take a constant $\sigma>0$.  Then we have
$$
-\int\nabla u\nabla(u^{\sigma}\xi)=\int u^{\sigma+\tau}\xi.
$$
Using integration by part, we know that the left side is
$$
-\int\nabla u\nabla(u^{\sigma}\xi)=-\sigma\int u^{\sigma-1}|\nabla
u|^2\xi -\int u^{\sigma}\nabla u\nabla\xi.
$$
Then $$ \sigma\int u^{\sigma-1}|\nabla u|^2\xi+\int
u^{\sigma+\tau}\xi=-\int u^{\sigma}\nabla u\nabla\xi.
$$ Note that
$$
-\int u^{\sigma}\nabla u\nabla\xi=\frac{1}{1+\sigma}\int
u^{1+\sigma}\Delta\xi\leq
\frac{C}{(1+\sigma)(R-\rho)^2}\int_{T_{R,\rho}} u^{1+\sigma},
$$
where $T_{R,\rho}=B_{R}-B_{\rho}$. Hence, we have
\begin{equation}\label{en12}
\sigma\int_{B_R} u^{\sigma-1}|\nabla
u|^2+\int_{B_R}u^{\sigma+\tau}\leq
\frac{C}{(1+\sigma)(R-\rho)^2}\int_{T_{R,\rho}} u^{1+\sigma}.
\end{equation}
Let first consider two cases.

(1). If we choose $\sigma=-\tau$, then we get
$$ -\tau\int_{B_R} u^{-\tau-1}|\nabla u|^2+|B_1(0)|R^n\leq
\frac{C}{(1-\tau)(R-\rho)^2}\int_{T_{R,\rho}} u^{1-\tau}.
$$
 (2). If we send $\sigma\to 0$ in (\ref{en12}), then we get
$$
\int_{B_R}u^{\tau}\leq \frac{C}{(R-\rho)^2}\int_{T_{R,\rho}} u.
$$

In the following, we do iteration. Let $\sigma=-\tau+p$ in
(\ref{en12}). Then we have
\begin{equation}\label{en121}
(-\tau+p)\int_{B_R} u^{-\tau+p-1}|\nabla u|^2+\int_{B_R}u^{p}\leq
\frac{C}{(1-\tau+p)(R-\rho)^2}\int_{T_{R,\rho}} u^{1-\tau+p}.
\end{equation}
Then we have
$$
\frac{4(-\tau+p)}{(-\tau+p+1)^2}\int_{B_R}|\nabla(u^{\frac{-\tau+p+1}{2}})|^2+
\int_{B_R}u^{p}\leq
\frac{C}{(1-\tau+p)(R-\rho)^2}\int_{T_{R,\rho}} u^{1-\tau+p}.
$$
We now in the standard Nash-Moser iteration situation. Hence, for
any $0<q$, we have a uniform constant $C(q,n)$ such that
$$
\sup_{B_{\theta R}} u\leq
\frac{C(q,n)}{((1-\theta)R)^{n/q}}(\int_{B_R} u^q)^{1/q}.
$$
Once we have a upper bound, we can use the standard
Calderon-Zugmund Lp theory to conclude that $u$ is a smooth
solution.
\end{section}

\begin{section}{Finite Morse Index  solutions}\label{FINI}

Assume that $u$ is the positive solution  to (\ref{Morse}) with
finite Morse index $k$. Choose a large ball $B_R(0)$ which
contains the supports of all $\phi_j$'s. Let
$$T_r=B_{R+1+r}(0)-B_{R+1}(0).$$ Then we have
\begin{equation}\label{stab}
E(\phi)\geq 0
\end{equation}
for all $\phi\in C^{\infty}_0(T_r)$. Let $\xi$ be a smooth cut-off function with compact support in $ T_r$.  Let
$\phi=u^{-q}\xi$. Then we have the following stability condition
for any $\epsilon>0$,
\begin{eqnarray*}
&&(-\tau)\int u^{-2q-1+\tau}\xi^2\\
&\leq & \int
|u^{-1}D\xi-qu^{-q-1}\xi Du|^2 \\
&\leq &(1+\frac{|q|}{2\epsilon})\int
u^{-2q}|D\xi|^2+(q^2+2|q|\epsilon)\int u^{-2q-2}\xi^2|Du|^2.
\end{eqnarray*}

Using the weak form of the equation (\ref{Morse}) with the testing
function $\xi^2u^{-\beta}$,$\beta=2q+1>0$, we have
$$
\beta\int u^{-\beta-1}\xi^2|Du|^2\leq \int
u^{-\beta+\tau}\xi^2+2\int u^{-\beta}\xi|Du||D\xi|,
$$
and then we have, using the Cauchy-Schwartz inequality, for any
$\delta>0$,
$$
(\beta-2\delta)\int u^{-\beta-1}\xi^2|Du|^2\leq\int
u^{-\beta+\tau}\xi^2+\frac{1}{2\delta}\int u^{-\beta+1}|D\xi|^2.
$$
Inserting this into the stability condition we get
$$
(-\tau)\int u^{-2q-1+\tau}\xi^2\leq C(\epsilon,\delta, q)\int
u^{-2q}|D\xi|^2+(\frac{q^2+2|q|\epsilon}{2q+1-2\delta})\int
u^{-2q-1+\tau}\xi^2.
$$
Choose $-\frac{1}{2}<q <-\tau+\sqrt{\tau^2-\tau}$ and $\epsilon$, $\delta$ small
enough depending on $q$, we can have
$$
\frac{q^2+2|q|\epsilon}{2q+1-2\delta}< -\tau.
$$
Hence, for some constant $C(\tau, q)$, we have
$$
\int u^{-2q-1+\tau}\xi^2 \leq C(\tau, q)\int u^{-2q}|D\xi|^2.
$$
Take $q>0$ and replace $\xi$ by $\xi^{q+\frac{1-\tau}{2}}$ to get
$$
\int (\frac{\xi}{u})^{2q+1-\tau}\leq C(\tau, q)\int
(\frac{\xi}{u})^{2q}\xi^{-1-\tau}|D\xi|^2.
$$
Here $C(\tau,q)$ is another constant. Using the Young inequality
$$
ab\leq \frac{(\epsilon a)^{\alpha}}{\alpha}
+\frac{\alpha-1}{\alpha}(\frac{b}{\epsilon})^{\frac{\alpha}{\alpha-1}}
$$
with $\alpha=\frac{q+\frac{1-\tau}{2}}{q}$,
($\frac{\alpha}{\alpha-1}=\frac{2q}{1-\tau}+1$),
$a=(\frac{\xi}{u})^2$, and $b=(\xi^{-\frac{1+\tau}{2}}|D\xi|)^2$,
and choosing $\epsilon$ small, we get that
\begin{equation}\label{eqn20}
\int (\frac{\xi}{u})^{2q+1-\tau}\leq C(\tau,
q)\int(\xi^{-\frac{1+\tau}{2}}|D\xi|)^{\frac{4q}{1-\tau}+2}.
\end{equation}
Take $0<q<-\tau+\sqrt{\tau^2-\tau}$ such that
$$ n\leq\frac{4q}{1-\tau}+2
$$ and then we can find that
$$
\int_{T_r}(\frac{1}{u})^{2q+1-\tau}\leq C(R,\tau,q),
$$
for all $r>0$. Note that the restriction of $q$ is
$$
\frac{(n-2)(1-\tau)}{4}\leq q<-\tau+\sqrt{\tau^2-\tau},
$$
which implies that
$$
n\leq 2+\frac{ 4}{1-\tau} ( -\tau+ \sqrt{\tau^2 -\tau} )
$$
Set $q$ such that $n=\frac{4q}{1-\tau}+2$ and $p=2q+1-\tau$. Then
$p>n$, and we use the lower bound of $u$ to get that
$$
\int_{B_r(0)}(\frac{1}{u})^{p}\leq C(R,\tau,q),
$$
for any $r>0$. Note that
$1-\frac{n}{p}=1-\frac{2}{1-\tau}=-\frac{\tau+1}{1-\tau}$. Using
our equation we find that
$$
|\Delta u|\in L^p(\mathbf{R}^n).
$$
Using the standard $L^p$ estimate we find that
\begin{Thm} Let $u$ is a positive solution with finite Morse
index on $\mathbf{R}^n$. We now assume that $u(x)\geq u(0)=1$.
Then,
$$ |Du(x)|\leq C|x|^{1-\frac{n}{p}}=C|x|^{-\frac{\tau+1}{1-\tau}},
$$
and then we have the growth estimate
$$
u(x)\leq C(1+|x|^{-\frac{2\tau}{1-\tau}}).
$$
\end{Thm}

In the case when the solution $u$ is stable, we can take
$T_r=B_r(0)$. In this case, we have the following result:
\begin{Thm}
There is no stable positive solutions to (\ref{Morse}) for $2\leq
n\leq 2+\frac{4}{1-\tau} (-\tau +\sqrt{\tau^2-\tau} )$.
\end{Thm}
\begin{proof} Assume not.
Let $\xi$ be a cut-off function such that $\xi=1$ on the ball
$B_R(0)$,
$$
\xi(x)=2-\frac{log|x|}{log R},
$$
for $x\in B_{R^2}(0)-B_R(0)$,and $\xi=0$ outside $B_{R^2}(0)$.
Then we get from the estimate (\ref{eqn20}) that
$$
\int_{B_R(0)} \frac{1}{u^n}\leq \frac{C}{(log R)^{n-1}}\to 0,
$$
which is impossible.
\end{proof}

We remark that the special case when $\tau=-1$ has been obtained
in \cite{Mea}.

We now prove Theorem \ref{morse}.
\begin{proof}
Using the test function $ \phi=\xi$ in (\ref{stab}), we obtain
that
\begin{equation}
\label{u1}
\int_{B_R (0)} u^{\tau-1} \leq C
\end{equation}
where $C$ is independent of $R>1$. In the following, we shall let
$\tau=-p$, which is often used in the literature about
super-linear elliptic equations.

We now perform the following scaling:
\begin{equation}
\label{sc1}
 u(r, \theta) = A_p r^{\frac{2}{p+1}} v(t, \theta), t = \log r, r =|x|
\end{equation}
where $A_p =(\frac{p+1}{2})^{ \frac{2}{p+1}}$.

Thus we obtain that $v (t, \theta)$ satisfies
\begin{equation}
\label{veqn}
  v_{tt} +\frac{4}{p+1} v_t + v_{\theta \theta}
  + \frac{4}{ (p+1)^2} v - \frac{4}{ (p+1)^2 v^p}  =0, t \in (-\infty, +\infty), \ \theta \in S^1
\end{equation}

We first claim

\begin{equation}
\label{cl1}
  v(t, \theta) \geq C \ \ \mbox{for}\ t >2.
\end{equation}

In fact, from (\ref{u1}) and (\ref{veqn}), we obtain that
\begin{equation}
\label{v2}
\int_{t}^{t+1} v^{-p-1} (t, \theta) ds d \theta \to 0, \ \mbox{as} \ t \to +\infty.
\end{equation}
Let $m= \frac{1}{v^{p+1}}$.  Then it is easy to see that $m$ satisfies
\begin{equation}
m_{tt}+\frac{4}{p+1} m_t + m_{\theta \theta} + C_1 m^2 \geq 0
\end{equation}

Let us fix a point $ {\bf x}_0=(t_0, \theta_0) \in (1, \infty) \times S^1$.
Set $\hat{m}= e^{\frac{2}{p+1} (t_0-t)} m$. Then $\hat{m}$ satisfies
\begin{equation}
\hat{m}_{tt}+\hat{m}_{\theta \theta} + C_2 \hat{m}^2 \geq 0
\end{equation}
for $ (t, \theta) \in [ t_0-1, t_0+1] \times S^1$.

By Lemma 2.2 of \cite{hw} (see also Theorem 1.7 in \cite{And}),
 we see that there exists $\eta_0>0$ such that for any $r>0$ if $\int_{B_r {\bf x}_0 } \hat{ m} dx \leq \eta_0$,
then
$$\hat{ m}(t, \theta) \leq \frac{C}{r^2} \int_{B_r ({\bf x}_0} {\hat{ m}} (x) dx \;\;
\mbox{for $ (t, \theta) \in B_{r/2} ({\bf x}_0) $}$$

Choosing $t_0>8$ large enough so that
\begin{equation}
\label{v3} \int_{t}^{t+1} v^{-p-1} (t, \theta) d \theta < e^{-8}
\eta_0, \ \mbox{for} \ t >\frac{t_0}{2}.
\end{equation}

Then
\begin{equation}
\label{v3}
\int_{t}^{t+1} \hat{m} (t, \theta) ds d \theta < \frac{1}{2} \eta_0, \ \mbox{for} \ t >\frac{t_0}{2}.
\end{equation}
Thus
$$ \hat{m} (t, \theta) \leq C$$
for $(t, \theta) \in B_{\frac{1}{2}} ( (t_0, \theta_0))$, which is equivalent to that $ v(t, \theta) \geq C$.

(\ref{cl1}) implies that  $ v(t, \theta) \geq C$. On the other
hand, it is easy to see that by the Harnack inequality, $ v(t,
\theta) \leq C$. By the results of L. Simon \cite{si}, $ v(t,
\theta) \to v(\theta)$, where $ v(\theta)$ satisfies
\begin{equation}
v_{\theta \theta} +\frac{4}{ (p+1)^2} v -\frac{1}{  v^p} =0, \ \mbox{$v$ \ is $2\pi$-periodic}.
\end{equation}
By Theorem 2.1 of \cite{cw}, $v(\theta) \equiv constant$ if $ p
\not = 3$, and $  v(\theta)= (\lambda \cos^2 \theta + \lambda^{-1}
\sin^2 \theta)^{1/2}$ for $p=3$.  This implies
\begin{equation}
\lim_{r \to +\infty} |x|^{-\frac{2}{p+1}} u(x) \geq \frac{2\mu}{p}
\end{equation}
for some $ \mu>0$.

Next, by explicitly solving the equation (it is an Euler equation), one finds
that any non-trivial solution of
\begin{equation}
\label{5.3}
 -k''-\frac{1}{r}k'-(\mu/r^2)k=0
\end{equation}
has infinitely many (and unbounded) positive zeros if $\mu>0$. (Note
that under the changes: $r=e^s$ and ${\tilde k}(s)=k(r)$, we see
that ${\tilde k}(s)$ satisfies the equation
$${\tilde k}''(s)+\mu {\tilde k}(s)=0. $$
It is easily seen that ${\tilde k}(s)$ has infinitely many positive
zeroes for any $\mu>0$.) Thus, we can easily deduce that $q$ has
infinitely many positive zeros. Our claim holds.

We denote the zeroes of $k$ as $ 0<r_1 <r_2 <...< r_k <...$
where $r_k \to +\infty$ as $k \to +\infty$. Let $k_0$ be such that
\begin{equation}
 \frac{p}{u^{p+1}} \geq \frac{2 \mu}{ r^2}, r>r_k, k \geq k_0
\end{equation}

We now in the position to complete the proof of Theorem \ref{morse}.
Let $N>0$ be fixed and $ i \geq k_0$. Let $h_i$ be the function defined to be
$k(|x|)$ for $|x|$ between the $i$th and $(i+1)$th the
zeros of $k$ and to be zero otherwise. Then $h_i \in
H^1(R^2)$, $h_i$ are orthogonal (in $L^2(R^2)$ or $H^1 (R^2)$)
and by multiplying (\ref{5.3}) by $h_i$ and integrating between
these zeros we see that
$$Q(h_i)=\int_{R^2} \Big[ |\nabla h_i|^2-\frac{p}{u^{p+1}} h_i^2 \Big]
$$
\[ =\int_{R^2} \Big[ \frac{\mu}{r^2}   -\frac{p}{u^{p+1}}\Big] h_i^2\]
is strictly negative at each $h_i$. Hence the span of $h_i$ is an
$(N-1)$-dimensional subspace of $C_0^\infty (R^2)$ such that
$Q(h)<0$. Since
$h_i$ has compact support it follows easily that there is an
$(N-1)$-dimensional subspace of $H^1 (R^2)$ such that
$$\int_{R^2} [|\nabla h(y)|^2-\frac{p}{u^{p+1}} h^2 ] <0
$$
and hence the Morse index of $u$ must be at least $N$.
Since $N$ is arbitrary, the Morse index of $u$ is infinity, a contradiction to our assumption.

\end{proof}

\end{section}

\begin{section}{Existence theory}\label{EXIST}
We now consider the existence problem of the problem (\ref{eqn1}).
One can easily to find the radial solutions to (\ref{eqn1}) in the
whole space (see Theorem 3.2 in \cite{Mea} for the case when
$\tau=-1$). In general, we have

\begin{Prop}\label{thm22} Assume $(-1\not=)\tau\leq 0$,
and let $\Omega\subset \mathbf{R}^n$ be a bounded smooth domain.
Given any smooth positive boundary data $\phi$. Assume that
$\underline{u}$ is a sub-solution to (\ref{eqn1}) with
$\underline{u}\leq \phi$ on $\partial\Omega$. Then there is smooth
positive solution to (\ref{eqn1}) with boundary data $\phi$.
\end{Prop}

We remark that in the case when $\tau=-1$, an existence result has
been discussed in \cite{Mea} by degree argument. Since the proof
of Theorem \ref{thm22} is simple, we give it here.
\begin{proof} Choose a large constant $M$ such that
$\overline{u}=M$ is a super-solution. Then one can use the
standard super-sub solution method to get a positive solution.

We give here the variational method. Let
$$\mathbf{A}=\{u\in
H^1(\Omega);\underline{u}\leq u\leq M \; in \; \Omega, u=\phi \;
in\;
\partial\Omega\}.$$
 Define
 \begin{equation}\label{energy}
I(u)=\frac{1}{2}\int |Du|^2-\frac{1}{1+\tau}\int |u|^{1+\tau}.
\end{equation}
It is easy to see that $I(\cdot)$ is bounded from below on the
closed convex set $\mathbf{A}$. Since for any $u,w\in\mathbf{A}$,
$$
|\int u^{1+\tau}-w^{1+\tau}|\leq C\int |u-w|.
$$
(We get this by mean value theorem in Calculus). Then we can use
the Sobolev compactness imbedding theorem to get a minimizer $u$
of the functional $I(\cdot)$ in the set $\mathbf{A}$, which is the
solution to (\ref{eqn1}) with the boundary data $\phi$ (see
\cite{St}).
\end{proof}

The advantage of variational methods is that one may prove that
the minimizer on the class $\mathbf{A}$ is a stable solution in
the usual sense. Since we shall not use this fact, we shall not
discuss it. The nature question is how to find a sub-solution on a
bounded domain. The usual way is to use one-dimensional (or any
lower dimensional) solution or radial solution on the whole space.
One can also choose a large ball containing the bounded domain,
and solve the equation on the ball to get radial solutions on the
ball. Such radial solutions are the sub-solutions to the equation
on the original domain if the boundary value of the radial
solutions on the domain are less than the given boundary data
$\phi$.

In particular, as an application of Theorem \ref{thm22}, we have
\begin{Coro}
 Assume $(-1\not=)\tau\leq 0$,
and let $\Omega\subset \mathbf{R}^n$ be a bounded smooth domain.
Given any positive smooth boundary data $\phi$ in
$\partial\Omega$. Assume that there is a radial solution $u(r)$ in
lower dimensional space $\mathbf{R}^k$ ($k<n$) or the whole space
$\mathbf{R}^n$
 such that
$\phi(x)>u(|x|)$ on $\partial\Omega$. Then there is smooth
positive solution to (\ref{eqn1}) with boundary data $\phi$.
\end{Coro}
\begin{proof} Here we need only to use $u(r)$ as a sub-solution to
(\ref{eqn1}) on the domain $\Omega$ in Proposition \ref{thm22}.
\end{proof}

Although the argument in the proof of this Corollary is simple, it
can be used to study the existence result of positive solutions
for a large class of singular elliptic partial differential
equations such as
$$
\Delta u+ au \log u=0
$$
with Dirichlet boundary data. One may see \cite{G93} and
\cite{Ma06} for related results.

\end{section}

\begin{section}{Consequences of Theorem \ref{thm4}}\label{CONS}

 As an easy consequence
of Theorem \ref{thm4}, we have
\begin{Prop}\label{Prop11}
Assume $\tau\leq 0$ and $\Omega\subset \mathbf{R}^n$. Let $f$ be
as in Theorem \ref{thm4} above. Let $u\in C^0(\Omega)$ be a
positive weak solution to the equation to (\ref{eqn10}) in
$\Omega$.
  Then for any $R>0$ and $x_0\in \Omega$ ( with $B_R(x_0)\subset \mathbf{R}^n$),
  we have absolute constant $C(n,\tau)$ such that
 \begin{equation}\label{max}
max_{\partial B_R(x_0)}u=sup_{B_R(x_0)}u\geq
C(n,\alpha)R^{\frac{2}{1-\tau}}.
\end{equation}
\end{Prop}
The proof of this is direct by using our $L1$ lower bound
(\ref{L1}) since $u$ is subharmonic and the maximum occurs only at
boundary point.

Proposition \ref{Prop11} immediately implies the following

\begin{Coro}
Assume $\tau\leq0$ and assume that $\Omega\subset \mathbf{R}^n$ is
an open subset in $\mathbf{R}^n$. Let
$f:\mathbf{R}\to\mathbf{R}_+$ be a positive function such that
$$
s^{\frac{\tau}{\tau-1}}f(s)^{\frac{1}{1-\tau}}\geq C_0, \; for \;
s>0
$$
for some constant $C_0$. Then there is a positive constant
$C=C(\Omega)$ such that if the positive boundary data $\phi\leq C$
on $\partial\Omega$, the Dirichlet problem to the equation
(\ref{eqn10}) in $\Omega$ with $u=\phi$ on the boundary
$\partial\Omega$ has no nontrivial positive weak $C^0$-solution.
\end{Coro}

In fact, we take a ball $B_R(x_0)$ in the domain $\Omega$ and let
$C=C(n,\alpha)R^{\frac{2}{1-\tau}}$. Then we have $sup_{\Omega}
u=sup_{\partial\Omega} \phi>C$.

Using this result, we can find a sequence $(u_j)$ of positive
solutions to (\ref{eqn1}) in the ball $B=B_1(0)$ such that $min
u_j\to 0+$ as $j\to +\infty$. Choose $C$ large enough, we can
solve (\ref{eqn1}) to get a unique positive radial solution $u$
with $u=C$ on the boundary $\partial B$. Then there is a constant
$M>0$ such that $|u|_{C^3}\leq M$. Let $\phi$ be as in the
Corollary above and let $\phi_t=tC+(1-t)\phi$, where $t\in [0,1]$.
Let $p>1$ and let $\delta>0$. Given a smooth function $u>0$ in
$W^{2,p}(B)$ with boundary data $\phi_t$. Consider the problem
$$
\Delta v=\frac{v}{u^{1-\tau}}, \; in \; B,
$$
with the boundary data $v=\phi_t$ on the boundary $\partial B$.
Set $T_t(u)=v$ and
$$\mathbf{U}(\delta)=\{v\in
W^{2,p}(B); v>\delta, |v|_{W^{2,p}(B)}\leq M(\delta)\}.$$ Here
$M(\delta)$ is a constant coming from the Lp estimate depending on
$\delta$. Then the fixed point of $T_t$ in $\mathbf{U}(\delta)$ is
a positive solution to (\ref{eqn1}) with boundary data $u=\phi_t$.
Note that $T_t$ is a compact operator from $\mathbf{U}(\delta)\to
W^{2,p}(B)$,
$$
deg(1-T_1,\mathbf{U}(\delta),0)=1
$$
and
$$
deg(1-T_0,\mathbf{U}(\delta),0)=0.
$$
Hence, we have $t_0\in (0,1)$ and $u\in\partial\mathbf{U}(\delta)$
such that $T_{t_0}(u)=u$.
 Choose $\delta=\delta_j\to
0$. Hence, we have a sequence $t_j\in(0,1)$ such that
$$
T_{t_j}(u_j)=u_j
$$
and $u_j\in \partial\mathbf{U}(\delta_j)$ (which implies that
$min_B u_j=\delta_j\to 0$).

We can do the same thing to (\ref{eqn1}) in the domain $\Omega$.
Hence, we have

\begin{Coro} \label{co12}
Given a bounded regular domain $\Omega$ in $\mathbf{R}^n$ such
that for sufficiently large constant $C>0$ as the Dirichlet
boundary data, the problem (\ref{eqn1}) has unique positive
solution. There exists a sequence of positive solutions $(u_j)$ to
(\ref{eqn1}) in the domain $\Omega$ with
$$
min_{\Omega} u_j\to 0.
$$
\end{Coro}
An open question is, which domain has the uniqueness property in
Corollary (\ref{co12}).

 Using Theorem \ref{thm4}, we can easily derive
the following:
\begin{Prop}\label{thm5} There is no positive solution to (\ref{eqn1})
in a cone-like unbounded domain $\Omega$ with the bound
$$
R^{-n-\frac{2}{1-\tau}+\sigma}\int_{B_R}udx\leq K
$$
for every ball $B_R\subset\Omega$ with $R\geq 1$, for some
constant $K>0$ and $\sigma>0$.
\end{Prop}
 \begin{proof} Assume we have a positive solution $u$. Note that we can choose an arbitrary large
 ball $B_R$ in the domain $\Omega$.
Then, using our Theorem \ref{thm4}, we get $$ R^{\sigma}\leq K,
$$
which is not true by sending $R\to+\infty$.
\end{proof}

We point out that for solution $u$ to (\ref{eqn1}), the quantity
$$ R^{-n-\frac{2p}{1-\tau}}\int_{B_R(0)}u^p
$$
is dimensionless. By this, we mean that if $u$ is a solution to
(\ref{eqn1}) in the ball $B_R(0)$, then the function $$
v(x)=R^{-\frac{2}{1-\tau}}u(Rx)
$$ is a solution to (\ref{eqn1}) in $B_1(0)$ with
$$R^{-n-\frac{2}{1-\tau}p}\int_{B_R}u^{p}dx=\int_{B_1(0)} v^p.
$$
Hence, in this sense, the $L^1$ lower bound estimate in Theorem
\ref{thm4} is the best one.
\end{section}

\begin{section}{Symmetry about Related equations}\label{mov1}

We also consider the integral equation
 \begin{equation}\label{eqn2}
  u(x)=h(x)-\int_{\mathbf{R}^{n}}|x-y|^{\mu-n}u(y)^{\tau}dy,
  \end{equation}
  with $n\geq2$, $0<\mu<n$, $h(x)$ is a positive smooth function, and $\tau<0$.
  This integral equation is closely
  related to the elliptic differential equation
the elliptic equation:
$$
(-\Delta)^{\mu/2}(u-h)=-u^{\tau}, \; \;in\; \; \mathbf{R}^n.
$$
It is clearly that the
   positive solution to (\ref{eqn2}) is bounded from above by $h$.

We shall use the following notation. Given any hyperplane $\pi$ in
$\mathbf{R}^n$. For any point $x\in \mathbf{R}^n$, let $x^{\pi}$
be the reflection of $x$ about the plane $\pi$ and let
$\pi(x)\in\pi$ be projection of $x$ into $\pi$. Define
$$
u^{\pi}(x)=u(x^{\pi})
$$
for any function $u:\mathbf{R}^n\to \mathbf{R}$.

In the famous paper of Gidas and Spruck \cite{GS81}, they proved
that for $n>2$ and $\mu=2$, and $1\leq \tau <\frac{n+2}{n-2}$, the
only non-negative solution to the equation
$$
  -\Delta u=u^{\tau},\ \ in \ \ {\mathbf{R}^n},
$$
 is zero. However, the negative index
$\tau$ case has not been treated before X.W.Xu's recent work
\cite{Xu07} . So, the following result can be considered as a
generalization of their result to the equation (\ref{eqn2}).

  \begin{Thm}\label{thm1} Given some $\beta>1$ and $q>1$.
  Let
  \begin{equation}
  u^{\tau-1}\in L^{\beta}(\mathbf{R}^{n}),  \label{cond1}
  \end{equation} be a positive
  solution of equation (\ref{eqn2}) with
  \begin{center}
  $
  \begin{array}{lll}
  \tau<0  & and &
  \beta=\frac{\tau-1}{\frac{n-\mu}{n}\tau-1}>\frac{2n}{n-\mu}.
  \end{array}
  $
  \end{center}
Assume that for some plane $\pi$, we have $h(x)=h({\pi}(x))$.
  Then $u(x)$ is symmetric to the plane $\pi$.
  \end{Thm}

Using the Kelvin transformation, we have another kind of result.

\begin{Thm}\label{thm2} Assume that $h(x)=h$ is a constant function.
Given some $\beta>1$ and $q>1$. Assume that
\begin{center}
  $
  \begin{array}{lll}
  \tau<0  & and &
  \beta=\frac{\tau-1}{\frac{n-\mu}{n}\tau-1}>\frac{2n}{n-\mu}.
  \end{array}
  $
  \end{center}

Then, for any positive solution to (\ref{eqn2}) with
  \begin{equation}
  u^{\tau-1}\in L^{\beta}(\mathbf{R}^{n}),  \label{cond1}
  \end{equation}  and
\begin{equation}
 |x|^{\mu-n}u(\frac{x}{|x|^{2}})\in
L^q(\mathbf{R}^n), \label{cond3}
\end{equation}
$u$ is radial symmetric at zero.
  \end{Thm}
  The proof of Theorem \ref{thm2} will be given in the last section.

In recent years, there are important progress in the study of
symmetry properties of non-negative solutions to Yamabe type
equations. In particular, X.Xu \cite{Xu07} has obtained some
related results to ours. His equation is the following
\begin{equation*}
  u(x)=\int_{\mathbf{R}^{n}}|x-y|^{\mu-n}u(y)^{\tau}dy,
  \end{equation*}
which corresponds to the elliptic equation:
$$
(-\Delta)^{\mu/2}u=u^{\tau}, \; \;in\; \; \mathbf{R}^n.
$$
 One should be caution about the negative sign before the Laplacian operator,
 which makes the equation can not have any positive solution on the whole space.
 We will use a symmetry method (see also \cite{CLO03} and
\cite{CM049}) to prove our results. This symmetry method is
powerful in our case since we can use the behavior at infinity of
the solution.

\emph{We now give a proof of Theorem \ref{thm1}}: After using a
rotation, we may assume that the hyperplane $\pi$ is orthonormal
to $x_1$ axis at the origin. So we may assume that $h(x)=h$ is a
constant in the following argument.

  For a given real number $\lambda$, we define
  $$
  \Sigma_{\lambda}=\{x=(x_{1},\cdots,x_{n})|x_{1}\geq\lambda\},
  $$
  and let $x^{\lambda}=(2\lambda-x_{1},x_{2},\cdots,x_{n})$ and
  $u_{\lambda}(x)=u(x^{\lambda})$.

    We can easily get the following
  \begin{Lem}\label{lem1}
  For any positive solution $u(x)$ of (1), we have
  \begin{equation}
  u_{\lambda}(x)-u(x)=-
  \int_{\Sigma_{\lambda}}(\frac{1}{|x-y|^{n-\mu}}-\frac{1}{|x^{\lambda}-y|^{n-\mu}})
  (u_{\lambda}(y)^{\tau}-u(y)^{\tau})dy. \label{int07}
  \end{equation}
  \end{Lem}
  \begin{proof}
  Let
  $$
  \Sigma_{\lambda}^{c}=\{x=(x_{1},\cdots,x_{n})|x_{1}<\lambda\}.
  $$
  Then it is easy to see that
  \begin{center}
  $
  \begin{array}{lll}
  h-u(x) & = & \int_{\Sigma_{\lambda}}\frac{1}{|x-y|^{n-\mu}}u(y)^{\tau}dy
  +\int_{\Sigma_{\lambda}^{c}}\frac{1}{|x-y|^{n-\mu}}u(y)^{\tau}dy\\
  & = & \int_{\Sigma_{\lambda}}\frac{1}{|x-y|^{n-\mu}}u(y)^{\tau}dy
  +\int_{\Sigma_{\lambda}}\frac{1}{|x-y^{\lambda}|^{n-\mu}}u(y^{\lambda})^{\tau}dy\\
  & = & \int_{\Sigma_{\lambda}}\frac{1}{|x-y|^{n-\mu}}u(y)^{\tau}dy
  +\int_{\Sigma_{\lambda}}\frac{1}{|x^{\lambda}-y|^{n-\mu}}u_{\lambda}(y)^{\tau}dy.
  \end{array}
  $
  \end{center}
  Here we have used the fact that
  $|x-y^{\lambda}|=|x^{\lambda}-y|$. Substituting $x$ by $x^{\lambda}$, we get
  $$
  h-u(x^{\lambda})=\int_{\Sigma_{\lambda}}\frac{1}{|x^{\lambda}-y|^{n-\mu}}u(y)^{\tau}dy
  +\int_{\Sigma_{\lambda}}\frac{1}{|x-y|^{n-\mu}}u_{\lambda}(y)^{\tau}dy.
  $$
  Thus
  $$
  u(x^{\lambda})-u(x)=-\int_{\Sigma_{\lambda}}(\frac{1}{|x-y|^{n-\mu}}-\frac{1}{|x^{\lambda}-y|^{n-\mu}})
  (u_{\lambda}(y)^{\tau}-u(y)^{\tau})dy.
  $$
  This implies (\ref{int07}).
  \end{proof}

We shall need the following
  Hardy-Littlewood-Sobolev inequality (see, for example, \cite{L83})
  \begin{equation}
  \|\int V(x,y)f(y)dy\|_{q}\leq P_{p,n}\|f\|_{p}
  \label{hls}
  \end{equation}
  with $V(x,y)=|x-y|^{-\nu}$ and
  $$1/p+\nu/n=1+1/q.$$

  {\bf Proof of Theorem \ref{thm1}.}

Define
  $$
  \Sigma_{\lambda}^{-}=\{x|x\in\Sigma_{\lambda},u(x)\geq
  u_{\lambda}(x)\},
  $$
  and
  $$
\Sigma_{\lambda}^{+}=\{x|x\in\Sigma_{\lambda},u(x)<
  u_{\lambda}(x)\}.
  $$
  Then
  $$
\Sigma_{\lambda}=\Sigma_{\lambda}^{+}\bigcup\Sigma_{\lambda}^{-}
  $$
  We want to show that for sufficiently positive values of
  $\lambda$, $\Sigma_{\lambda}^{-}$ must be empty.

  Note that for $y\in \Sigma_{\lambda}^{-}$, we have
  $u(y)^{\tau}\leq u_{\lambda}(y)^{\tau}$.
  Whenever $x,y\in\Sigma_{\lambda}$, we have that $|x-y|\leq|x^{\lambda}-y|$
  and
  $$
|x-y|^{\mu-n}\geq|x^{\lambda}-y|^{\mu-n}.
  $$
   Then by Lemma \ref{lem1}, for any $x\in\Sigma_{\lambda}^{-}$,
  \begin{equation}
  \begin{array}{lll}
  u(x)-u_{\lambda}(x) & = & -\int_{\Sigma_{\lambda}}(|x-y|^{\mu-n}-|x^{\lambda}-y|^{\mu-n})
  (u(y)^{\tau}-u_{\lambda}(y)^{\tau})dy\\
  & \leq & \int_{\Sigma_{\lambda}^{-}}|x-y|^{\mu-n}(u_{\lambda}(y)^{\tau}-u(y)^{\tau})dy\\
  & \leq & -\tau\int_{\Sigma_{\lambda}^{-}}|x-y|^{\mu-n}[u_{\lambda}^{\tau-1}(u-u_{\lambda})](y)dy.
  \end{array}
  \end{equation}

  It follows first from inequality (\ref{hls}) and the Holder inequality
  that, for any $q>n/(n-\mu)$,
\begin{equation}
\begin{array}{lll}
  \|u_{\lambda}-u\|_{L^{q}(\Sigma_{\lambda}^{-})}
  & \leq & C\|\int_{\Sigma_{\lambda}^{-}}|x-y|^{\mu-n}
  [u_{\lambda}^{\tau-1}(u_{\lambda}-u)](y)dy\|_{L^{q}(\Sigma_{\lambda}^{-})}\\
  & \leq & C(\int_{\Sigma_{\lambda}^{-}}u_{\lambda}(y)^{(\tau-1)\beta}dy)^{1/\beta}\|u_{\lambda}-u\|
  _{L^{q}(\Sigma_{\lambda}^{-})}\\
  & \leq & C(\int_{\Sigma_{\lambda}^{-}}
  u_{\lambda}(y)^{(\tau-1)\beta}dy)^{1/\beta}\|u_{\lambda}-u\|
  _{L^{q}(\Sigma_{\lambda}^{-})}.\,\,\,\,\,\,\,\\
  \end{array}
   \label{ineq}
\end{equation}

  By condition (\ref{cond1}), we can choose $N$
  sufficiently large, such that for $\lambda>N$, we have
  $$
  C(\int_{\Sigma_{\lambda}}u_{\lambda}(y)^{(\tau-1)\beta}dy)^{1/\beta}\leq\frac{1}{2}.
  $$
  Now (\ref{ineq}) implies that
  $$
  \|u_{\lambda}-u\|_{L^{q}(\Sigma_{\lambda}^{-})}=0,
  $$
  and therefore $\Sigma_{\lambda}^{-}$ must be measure zero, and
  hence empty. Then using the standard moving plane trick (see (\cite{CLO03}) and \cite{CM049}),
   we know
   that the solution $u$ is symmetric in the variable $x_1$.

So we complete the proof of Theorem \ref{thm1}.

   Finally, we remark that in some cases, using the analysis of ODE, we see that
   there is no radially symmetric
   positive solution to (\ref{eqn2}) with the condition (\ref{cond1}).
   However, we shall not do this here.
\end{section}

\begin{section}{proof of theorem \ref{thm2}}\label{mov2}

 Let us now introduce the Kelvin type
transform of $u$
  as follow
  $$
  v(x)=|x|^{\mu-n}u(\frac{x}{|x|^{2}})
  $$
  for any $x\neq0$. Then by elementary calculations, one can see
  that
  (\ref{eqn2}) is transformed into the following:
  \begin{equation}\label{eqnv1}
 h|x|^{\mu-n}-v(x)=\int_{\mathbf{R}^{n}}|x-y|^{\mu-n}|y|^{-\alpha}v(y)^{\tau}dy,
  \end{equation}
  where $\alpha=(n+\mu)-(n-\mu)\tau>0$. Obviously, $v(x)$ has a
  singularity at origin. Since $u$ is locally bounded, it is
  easy to see that $v(x)$ has no singularity at infinity, i.e.,
  for any domain $\Omega$ that is a positive distance away from
  the origin,
  \begin{equation}
  \int_{\Omega}v^{\beta}(y)dy<\infty. \label{int08}
  \end{equation}
  In fact, we have
  \begin{center}
  $
  \begin{array}{lll}
  \int_{\Omega}v^{\beta}(y)dy & = &
  \int_{\Omega}(|y|^{\mu-n}u(\frac{y}{|y|^{2}}))^{\beta}dy\\
  & = & \int_{\Omega^{\star}}(|z|^{n-\mu}u(z))^{\beta}|z|^{-2n}dz\\
  & = & \int_{\Omega^{\star}}|z|^{\beta(n-\mu)-2n}u(z)^{\beta}dz\\
  & \leq & C\int_{\Omega^{\star}}u(z)^{\beta}dz\\
  & < & \infty.
  \end{array}
  $
  \end{center}
  For the second equality, we have made the transform $y=z/|z|^{2}$.
  Since $\Omega$ is a positive distance away from
  the origin, $\Omega^{\star}$, the image of $\Omega$ under this
  transform, is bounded. Also note that $\beta(n-\mu)-2n>0$. Then we
  get the estimate (\ref{int08}).

  For a given real number $\lambda$, we define, as before,
  $$
  \Sigma_{\lambda}=\{x=(x_{1},\cdots,x_{n})|x_{1}\geq\lambda\},
  $$
  and let $x^{\lambda}=(2\lambda-x_{1},x_{2},\cdots,x_{n})$ and
  $v_{\lambda}(x)=v(x^{\lambda})$.

    We can easily get the following
  \begin{Lem}\label{lem2}
  For any solution $v(x)$ of (\ref{eqnv1}), we have
 \begin{eqnarray}\label{eqnv2}
  &&v_{\lambda}(x)-v(x)=h(|x^{\lambda}|^{\mu-n}-|x|^{\mu-n})\\
  &-&\int_{\Sigma_{\lambda}}(\frac{1}{|x-y|^{n-\mu}}-\frac{1}{|x^{\lambda}-y|^{n-\mu}})
  (\frac{1}{|y^{\lambda}|^{\alpha}}v_{\lambda}(y)^{\tau}-\frac{1}{|y|^{\alpha}}v(y)^{\tau})dy.\nonumber
 \end{eqnarray}
  \end{Lem}
  \begin{proof} The proof is similar to Lemma \ref{lem1}.
  Let
  $$
  \Sigma_{\lambda}^{c}=\{x=(x_{1},\cdots,x_{n})|x_{1}<\lambda\}.
  $$
  Then it is easy to see that
  \begin{center}
  $
  \begin{array}{lll}
  h|x|^{\mu-n}-v(x) & = & \int_{\Sigma_{\lambda}}\frac{1}{|x-y|^{n-\mu}}\frac{1}{|y|^{\alpha}}v(y)^{\tau}dy
  +\int_{\Sigma_{\lambda}^{c}}\frac{1}{|x-y|^{n-\mu}}\frac{1}{|y|^{\alpha}}v(y)^{\tau}dy\\
  & = & \int_{\Sigma_{\lambda}}\frac{1}{|x-y|^{n-\mu}}\frac{1}{|y|^{\alpha}}v(y)^{\tau}dy
  +\int_{\Sigma_{\lambda}}\frac{1}{|x-y^{\lambda}|^{n-\mu}}\frac{1}{|y^{\lambda}|^{\alpha}}v(y^{\lambda})^{\tau}dy\\
  & = & \int_{\Sigma_{\lambda}}\frac{1}{|x-y|^{n-\mu}}\frac{1}{|y|^{\alpha}}v(y)^{\tau}dy
  +\int_{\Sigma_{\lambda}}\frac{1}{|x^{\lambda}-y|^{n-\mu}}\frac{1}{|y^{\lambda}|^{\alpha}}v_{\lambda}(y)^{\tau}dy.
  \end{array}
  $
  \end{center}
  Here we have used the fact that
  $|x-y^{\lambda}|=|x^{\lambda}-y|$. Substituting $x$ by $x^{\lambda}$, we get
  \begin{eqnarray*}
  h|x^{\lambda}|^{\mu-n}-v(x^{\lambda})&=&\int_{\Sigma_{\lambda}}\frac{1}{|x^{\lambda}-y|^{n-\mu}}\frac{1}{|y|^{\alpha}}v(y)^{\tau}dy
  \\
  &+&\int_{\Sigma_{\lambda}}\frac{1}{|x-y|^{n-\mu}}\frac{1}{|y^{\lambda}|^{\alpha}}v_{\lambda}(y)^{\tau}dy.
  \end{eqnarray*}
  Thus
  \begin{eqnarray*}
  v(x)-v(x^{\lambda})&=&h(|x^{\lambda}|^{\mu-n}-|x|^{\mu-n})\\
  &-&\int_{\Sigma_{\lambda}}(\frac{1}{|x-y|^{n-\mu}}-\frac{1}{|x^{\lambda}-y|^{n-\mu}})
  (\frac{1}{|y^{\lambda}|^{\alpha}}v_{\lambda}(y)^{\tau}-\frac{1}{|y|^{\alpha}}v(y)^{\tau})dy.
  \end{eqnarray*}
  This implies (\ref{eqnv2}).
  \end{proof}

We shall also need the following doubly weighted
  Hardy-Littlewood-Sobolev inequality of Stein
  and Weiss (see, for example, \cite{L83})
  \begin{equation}
  \|\int V(x,y)f(y)dy\|_{q}\leq
  P_{\alpha,\beta,p,\nu,n}\|f\|_{p} \label{hls2}
  \end{equation}
  with $V(x,y)=|x|^{-\beta}|x-y|^{-\nu}|y|^{-\alpha},$
  $0\leq\alpha<n/p^{'}$, $0\leq\beta<n/q$, $1/p+1/{p'}=1$, and
  $$1/p+(\nu+\alpha+\beta)/n=1+1/q.$$

  {\bf Proof of Theorem 2.}  Define
  $$
  \Sigma_{\lambda}^{-}=\{x|x\in\Sigma_{\lambda},v(x)<v_{\lambda}(x)\},
  $$
  and
  $$
\Sigma_{\lambda}^{+}=\{x|x\in\Sigma_{\lambda},v(x)\geq
v_{\lambda}(x)\}.
  $$
  We want to show that for sufficiently negative values of
  $\lambda$, $\Sigma_{\lambda}^{-}$ and $ \Sigma_{\lambda}^{+}$ must be empty.

  Whenever $x,y\in\Sigma_{\lambda}$, we have that $|x-y|\leq|x^{\lambda}-y|$.
  Moreover, since $\lambda<0$,
  $|y^{\lambda}|\geq|y|$ for any $y\in\Sigma_{\lambda}$.
  Then by Lemma \ref{lem2}, for any $x\in\Sigma_{\lambda}^{-}$,
   \begin{equation}
  \begin{array}{lll}
  v_{\lambda}(x)-v(x) & \leq & -\int_{\Sigma_{\lambda}}
  (|x-y|^{\mu-n}-|x^{\lambda}-y|^{\mu-n})
  |y|^{-\alpha}(v_{\lambda}(y)^{\tau}-v(y)^{\tau})dy\\
  & \leq & -\int_{\Sigma_{\lambda}^{-}}|x^{\lambda}-y|^{\mu-n}
  |y|^{-\alpha}(v_{\lambda}(y)^{\tau}-v(y)^{\tau})dy\\
  & \leq & -\tau\int_{\Sigma_{\lambda}^{-}}|x-y|^{\mu-n}|y|^{-\alpha}
  [v^{\tau-1}(v_{\lambda}-v)](y)dy.\label{int9}
  \end{array}
   \end{equation}
  It follows first from inequality (\ref{hls2}) and then the Holder inequality
  that, for any $q>n/(n-\mu)$,
  which will be used below in the form that $\tau<\mu$,
 \begin{equation}
  \begin{array}{lll}
  \|v_{\lambda}-v\|_{L^{q}(\Sigma_{\lambda}^{-})}
  & \leq & C\|\int_{\Sigma_{\lambda}^{-}}|x-y|^{\mu-n}|y|^{-\alpha}
  [v^{\tau-1}(v_{\lambda}-v)](y)dy\|_{L^{p}(\Sigma_{\lambda}^{-})}\\
  & \leq & C(\int_{\Sigma_{\lambda}^{-}}v(y)^{(\tau-1)\beta}dy)^{1/\beta}\|v_{\lambda}-v\|
  _{L^{q}(\Sigma_{\lambda}^{-})}\\
  & \leq & C(\int_{\Sigma_{\lambda}}v(y)^{(\tau-1)\beta}dy)^{1/\beta}\|v_{\lambda}-v\|
  _{L^{q}(\Sigma_{\lambda}^{-})}.\,\,\,\,\,\,\,\\
  \end{array}\label{ineq1}
   \end{equation}

  By condition (\ref{int08}), we can choose $N$
  sufficiently large, such that for $\lambda>N$, we have
  $$
  C(\int_{\Sigma_{\lambda}}v(y)^{(\tau-1)\beta}dy)^{1/\beta}\leq\frac{1}{2}.
  $$
  Now (\ref{ineq1}) implies that
  $$
  \|v_{\lambda}-v\|_{L^{q}(\Sigma_{\lambda}^{-})}=0,
  $$
  and therefore $\Sigma_{\lambda}^{-}$ must be measure zero.
  Then we can use the moving plane trick as before to get that
  $v$ is symmetric at zero with respect to the $x-1$ direction.
 This completes the proof of
Theorem \ref{thm2}.

\end{section}

{\bf Acknowledgement:} We would like to thank Prof. C.Gui for
pointing out that in our proof of Theorem \ref{morse}, we can
allow the kernel of $E$ can have infinite dimension.

\end{document}